\documentclass[10pt]{amsproc}
\usepackage{amssymb,amsmath,amsthm,amsgen}
\usepackage{comment}
\usepackage{}

\newcommand{\old}[1]{}
\theoremstyle{plain}
\newtheorem{thm}{Theorem}[section]

\newtheorem{conj}{Conjecture}
\newtheorem{cor}[thm]{Corollary}

\newtheorem{prop}[thm]{Proposition}
\theoremstyle{definition}

\newtheorem{remark}[thm]{Remark}

\newtheorem{qn}[thm]{Question}

 at 10truept

\title[Partition Poset Homology]{Some Problems Arising from Partition Poset Homology\footnote{10 October 2014}}

\author{Sheila Sundaram }
\address{Pierrepont School, One Sylvan Road North, Westport, CT 06880}
\email{shsund@comcast.net}
\date{10 October 2014}

\dedicatory{Dedicated to Richard Stanley on the occasion of his 70th birthday.}

\thanks{The author would like to express her gratitude for the invitation to 
submit this paper, and to Patricia Hersh in particular for her encouragement.   }


\begin{document}

\begin{abstract}
We describe some open problems related to homology representations of subposets of the partition lattice, beginning with questions first raised in Stanley's work on group actions on posets.

\end{abstract}

\maketitle

\section{Introduction}




This paper is about a part of Richard Stanley's work that has exerted a tremendous influence on my 
mathematical career.  I am privileged to have had Richard as my thesis adviser.  I recall that during the first few months of our interaction 
 I was  in a somewhat catatonic state of awe;  this was  not exactly 
conducive to producing new mathematics.  All that changed when Richard taught 18.318 (M.I.T.'s yearly \lq\lq Topics in Combinatorics" course for graduate students)
 in the Spring of 1984.  
I forget  the exact subtitle of the course, but the emphasis was on symmetric functions and representation theory.

Richard's lectures were always a model of clarity and exposition.  He had a knack for making very difficult results seem obvious and effortless
(until I tried to reconstruct the arguments myself).  From Richard's 18.318 I learnt about Schur-Weyl duality, and what the mysterious   plethysm operation of 
symmetric functions meant in concrete representation-theoretic terms, for both the symmetric group and the general linear group. I still remember the sense of excitement  I felt upon finally 
gaining a useful understanding of these deep ideas.  I had already taken an earlier 18.318 offered by Phil Hanlon, 
on the character theory of the symmetric group, and I had read the first chapter of Ian Macdonald's 
 {\it Symmetric Functions and Hall Polynomials}.
 Richard's course, with his inimitable style, dry wit and understated humour, tied everything together for me, and gave me an appreciation of the endless possibilities, 
and the elegance, of the symmetric functions technique.  

Much of the material in that course has now been handed down to future generations in {\it Enumerative Combinatorics, Volume 2} (EC 2).  In addition to the lectures, 
the other feature that made Richard's courses invaluable was his problem sets (also handed down to posterity in both EC1 and EC2).  I remember working on them for hours on end, late into the night. 
They were truly addictive, although I was not successful in solving 
many of the problems.  Some are still open.  Today, after many years of my own teaching, I marvel at Richard's patience in reading all those papers 
(we were about ten in the class, I think). 
Many of my attempts resulted in only partial solutions, and yet he clearly read them all, and they were returned interspersed with comments here and there indicating a 
better direction to pursue.  It was always very exciting and tremendously encouraging to see a \lq\lq good" at the end of a solution, 
a couple of times \lq\lq very elegant proof," and once even 
\lq\lq I couldn't do this one myself." (For this one I have always suspected he was being a little too generous.)  In addition to Richard's remarkable patience and 
meticulousness in grading these problem sets, something I did not know at the time was that he also recorded what he deemed to be the best solutions, 
and generously and unfailingly credited his students for them in both EC 1 and EC 2.  

I believe it was here that I first understood how to dissect virtual modules via the Frobenius characteristic (something which  turned out to be very useful  in predicting {\it topological} properties of posets). Some years later, when I realised how inextricably the plethysm operation is  linked to the partition lattice, I would be very happy to be able to use what I had learnt in Richard's class, in my own research.
  That particular  18.318 of Richard's is, hands down, for me  the most   influential course I have ever taken.

Likewise, the most influential research paper that I read as a graduate student is Richard's {\it Some aspects of group actions on posets} \cite{St2}. It was at about the same time 
as Richard's course, but curiously, I would realise its profound impact   only when I  came back to it in a few years.  That paper introduced me to poset homology.
  The order complex of a partially ordered set provides a beautifully concrete way to illustrate  many of the big theorems in 
algebraic topology, and Richard's paper highlights the fascinating interplay between combinatorics, algebraic topology and representation theory. Although 
my thesis was on a  different subject, the large majority of my later work was on topics related to the partition lattice, rank-selected 
and other subposets, and homology representations.  

In this paper I will focus on questions raised in one section of Richard's {\it Some aspects of group actions on posets}, Section 7. This part alone immediately attracted a considerable amount of attention (\cite{CHR}, \cite{Ha1}, \cite{Ha2}) and  eventually spawned a vast body of research connecting the partition lattice to the free Lie algebra, subspace arrangements, configuration spaces, Lie operads, complexes of graphs  and various generalisations.  The comprehensive survey paper of Wachs \cite{Wa} is an excellent reference for the literature and for more recent developments.  Although extensively studied, partition posets  continue to be a source of interesting open problems, often with tantalising ramifications. The partition lattice and its subposets also serve as a guiding first example of many curious phenomena, both topological and representation-theoretic, that have since been generalised and rediscovered in many other contexts in  the last decade.    Richard foreshadows many of these developments in his paper with his characteristic uncanny insight.

\section{Partition lattice homology and the trivial representation}

Recall that $\Pi_n$ is the lattice of set partitions of an 
$n$-element set, ordered by refinement.  We say a block of 
a partition is nontrivial if it consists of more than one element (for this and other basic 
definitions see \cite{St3}, EC1).   For a bounded poset $P$ (i.e., one with a unique minimal element $\hat{0}$ and a unique maximal element $\hat{1}$) we denote 
by $\hat P$ the {\it proper part} of $P,$ i.e., 
the poset $P$ with the greatest element $\hat 1$ 
and the least element $\hat 0$ removed.  We write $\Delta(P)$ for 
the order complex of $P;$ the simplices of $\Delta(P)$ are the 
chains of $\hat P.$  By the $i$th (reduced) homology 
$\tilde{H}_i(P)$ of $P$ we mean the 
$i$th (reduced) simplicial homology of its order complex $\Delta(P).$  
All homology in this paper is reduced homology taken with integer coefficients 
except for representation-theoretic discussions, in which case 
we take coefficients over the complex field. All posets 
are bounded unless explicitly stated otherwise, although  we will always consider the proper part of the poset so as to avoid topologically trivial situations.

Stanley's paper \cite{St2} considers a finite  group $G$ acting in an  order-preserving fashion on a poset $P,$ 
and the resulting action of $G$ on the unique novanishing homology of $P$  in the special case when $P$ is Cohen-Macaulay, that is, 
 if the reduced 
homology of the order complex of the proper part of every interval $[x,y]$ of $P,$ 
$\hat 0\leq x\leq y\leq \hat 1,$ vanishes below the top dimension.  
(See \cite{St3} for definitions.)
If $P$ is a ranked poset, then   the rank-selected subposet $P_S$  of 
$P$ consists of all elements with rank belonging to  the subset $S.$ Stanley \cite{St1}, and independently 
Baclawski \cite{Ba1}, had shown that rank-selection preserves the Cohen-Macaulay property.  

We will confine our discussion to the case of rank-selection in the partition lattice $\Pi_n.$  Here the nontrivial ranks are $1,2,\ldots , n-2.$  It will be convenient to write  $[a,b]$ for the interval of consecutive ranks $\{a, a+1,\ldots, b\}.$  The automorphism group of the lattice is  the symmetric group 
$S_n,$ and  Stanley observed
that the rank-selected homology modules refine the permutation representation  of $S_n$ on the maximal chains of $\Pi_n$ according to the subsets $S$ of the 
ranks, i.e., of $\{1,2,\ldots, n-2\}.$  That is, if $\alpha_n$ denotes the permutation representation of $S_n$ on the maximal chains of $\Pi_n$ and $\beta_S(n)$ denotes 
the representation on the homology of the rank-selected subposet $\Pi_n(S),$ then 
\begin{equation}
\alpha_n=\sum_{S\subseteq [1,n-2]} \beta_S(n).
\end{equation} 
More generally,  (see \cite{St2}, Theorem 1.1), if $\alpha_S(n)$ denotes the permutation representation of $S_n$ on the maximal chains of $\Pi_n(S),$ one has
\begin{equation}
\alpha_S(n)=\sum_{T\subseteq S} \beta_T(n),
\end{equation} 
and hence, by an application of the Hopf-Lefschetz formula, 
\begin{equation}
\beta_S(n)=\sum_{T\subseteq S} (-1)^{|S-T|} \alpha_T(n)
\end{equation} 

Stanley asked for a characterisation of the homology representations $\beta_S(n).$ He also gave a complete description of $\beta_S(n)$ in the case when 
$S$ is the full set of ranks $[1,n-2]$, i.e, of the top homology module of $\Pi_n.$     Let ${\rm sgn}$ denote the sign representation of $S_n,$ and let $\psi_n$ be the $S_n$-representation obtained by inducing any    faithful irreducible representation of a cyclic subgroup 
of order $n.$  Write $\beta_n$ for $\beta_{[1,n-2]}(n),$ i.e., the top homology representation of $\Pi_n.$ Then, using a crucial M\"obius function computation of Hanlon (\cite{Ha1}), Stanley showed that 

\begin{thm} (\cite{St2}, Theorem 7.3 and Corollary 7.5) $$\beta_n= ({\rm sgn})\  \psi_n.$$   The restriction of $\beta_n$  to $S_{n-1}$ is the regular representation.
\end{thm}

As mentioned in the introduction, this observation would eventually generate a large literature on the partition lattice (see bibliography).
 
We note that the rank-selection question was settled completely by Stanley in \cite{St2} for the Boolean lattice of subsets of a set of size $n,$ the subspace lattice $Gl_n(q)$ and the lattice of faces of a cross-polytope.

Stanley's paper and the questions raised in it were the primary motivation for the paper \cite{Su1}.
Here rank-selection in $\Pi_n$ was studied more extensively, and 
recursive plethystic descriptions were given for  the rank-selected homology representations. Although a considerable amount of representation-theoretic and enumerative information can be extracted from these formulas (see \cite{Su1}, \cite{Su2}), the problem of giving a nice characterisation of these representations appears to be a difficult one.    There is, however, a fundamental  three-term plethystic  recurrence which can be used to compute the 
representations efficiently.  Let $h_n$ denote the homogeneous symmetric function of degree $n,$ and let $\beta_S(n)$ denote both the rank-selected representation of $S_n$ and its  Frobenius characteristic (which is thus a symmetric function of degree $n$).   The square brackets denote the plethysm operation.  (See \cite{M} and \cite{St4} for definitions.)   Likewise, let $\alpha_S(n)$ denote both the $S_n$-representation and its Frobenius characteristic on the maximal chains of $\Pi_n(S).$ Finally, let $S(n,k)$ denote the Stirling number of the second kind, that is, the number of set partitions with $k$ nonempty blocks.

\begin{thm} (\cite{Su1}, Theorem 2.13, Proposition 2.16, Proposition 3.1) Let $S=\{s_1<s_2<\ldots<s_r\}, \ r\geq 2$ be a subset of the nontrivial ranks $[1, n-2]$ of $\Pi_n,$ and let 
$S-s_1$ denote the subset $\{s_2-s_1, s_3-s_1,\ldots,s_r-s_1\}$ of ranks in $\Pi_{n-s_1}.$  Then 
$$\beta_S(n)+\beta_{S\backslash \{s_1\}}(n) = \beta_{S-s_1}(n-s_1)\big[\sum_{i\geq 1} h_i\big]\vert_{\deg n}.$$
Hence, if $d_S(n)$ denotes the unique nonvanishing Betti number for the order complex of $\Pi_n(S),$  one has the recurrence
$$d_S(n)+d_{S\backslash \{s_1\}}(n) = d_{S-s_1}(n-s_1)\  S(n,n-s_1).$$
For the permutation representation of $S_n$ on the rank-selected maximal chains, one has 
$$ \alpha_S(n)=\alpha_{S-s_1}(n-s_1)[\sum_{i\geq 1} h_i\big]\vert_{\deg n}.$$
\end{thm}

The second question raised by Stanley concerns the multiplicity $b_S(n)$ of the trivial representation in the rank-selected homology $\beta_S(n)$ of $\Pi_S$.  Hanlon had shown in \cite{Ha2} that the multiplicity $b_S(n)$ is zero for $S=[1,r], r\geq 1.$ This result is recovered, and more conditions on $S$ for vanishing multiplicity are given in \cite{Su1}, where the action restricted to $S_{n-1}$ is also studied.  The primary tool was Theorem 2.2 above. Hanlon and Hersh \cite{HH}
strengthened and proved some of the results and conjectures of \cite{Su1} using  completely different techniques, spectral sequences of filtered 
complexes and an intricate partitioning of the quotient complex $\Delta(\Pi_n)/S_n$ \cite{He}.  

Let $E_n$ be the $n$th Euler number defined by the generating function $$\sum_{n\geq 0} E_n \frac{x^n}{n!} =\tan x+\sec x.$$ 
($E_{2n-1}$ is the {\it tangent} number, and $E_{2n}$ is the {\it secant} number.)
It is well-known (\cite{St3}) that $E_n$ counts the alternating permutations in $S_n,$ i.e., those permutations $\sigma\in S_n$ such that $\sigma(1)>\sigma(2)<\sigma(3)>\ldots.$   There is a large literature  on the Euler numbers.  Stanley's paper  also contains the following  result (for which he gives a second proof in (\cite{St5}, Theorem 3.4)):

\begin{thm} (\cite{St2}, Theorem 7.7)  The multiplicity of the trivial representation in the 
$S_n$-action $\alpha_n$ on the maximal chains of $\Pi_n$ (which is also the total number of $S_n$-orbits) is the Euler number $E_{n-1}.$ 
\end{thm}  

As observed by Stanley,  the multiplicity $b_S(n)$ of the trivial representation in the rank-selected homology module $\beta_S(n)$  gives rise to a refinement of $E_{n-1}$ into nonnegative integers indexed by subsets of $[1,n-2].$  

\begin{thm} (\cite{St2})
$$E_{n-1}=\sum_{S\subseteq [1,n-2]} b_S(n).$$ 
\end{thm}

Interestingly, there is a second such refinement that occurs naturally when considering the homology representation:  the subsets of $[1,n-2]$ also index a refinement of the succeeding Euler number $E_n$ into  nonnegative integers.  Let $b_S'(n)$ denote the mutiplicity of the trivial representation of $S_{n-1}\times S_1$ in $\beta_S(n).$ Then 

\begin{thm} (\cite{Su1}, Proposition 3.4 (1) and p.269) The multiplicity of the trivial representation of $S_{n-1}\times S_1$ on the maximal chains of $\Pi_n$ is $E_n,$ and hence 
$$E_n=\sum_{S\subseteq [1,n-2]} b_S'(n).$$

\end{thm}

Tables of values of $b_S(n)$ and $b_S'(n)$ for $4\leq n\leq 9$ are given in (\cite{Su1}, pp. 286-288).  (The entries for $n=7$ and the 
subset $S=\{2,4,5\}$ appear 
to have been omitted: The missing entries are $b_7(\{2,4,5\})=5$ and $b_7'(\{2,4,5\})=23.$)  

\begin{qn} As far as we know, it is an open problem to describe combinatorially the refinements of Theorems 2.4 and 2.5. 
Is there a way to use more sophisticated topological techniques, by embedding quotient complexes, or using partitionings as in \cite{HH},  that would shed some light on them?  It seems curious that there should 
be two such combinatorial refinements of the Euler numbers.
\end{qn}

An elegant  unified generalisation of Stanley's Theorem 2.3, for each finite root system $R$ and corresponding Coxeter group $W$,  appears in a recent paper \cite{JV} of Josuat-Verg\`es,
who computes  the number $K(W)$ of $W$-orbits on the maximal chains of the intersection lattice $I$ of $R$.  In the particular case $R=A_{n-1},$  one has $I=\Pi_n$ and  $W=S_n,$ and thus by Theorem 2.3, Josuat-Verg\`es' number is 
\begin{equation} K(S_n)=K(\Pi_n)=E_{n-1}=\#\{ S_n\rm{-orbits\ on\ maximal\ chains\ of\ } \Pi_n\}.\end{equation}  
However,  decades earlier, 
Springer \cite{Sp}, in the same general setting,  had also defined and computed an integer $T(W)$ as follows: if $S$ is a set of simple roots for $R$, and $J\subset S,$ let $\sigma(J,S)$ denote the number of elements $w\in W$ such that $w\alpha>0$ for $\alpha\in J$ and $w\alpha<0$ for $\alpha\in S\backslash J.$  Then $T(W)$ is the maximum value of $\sigma(J,S)$ as $J$ ranges over all subsets of $S.$ 
  For $R=A_{n-1},$  Springer showed that $T(W)=E_n,$ and thus it transpires that his number \lq\lq matches" Theorem 2.5:
\begin{equation} T(S_n)=T(\Pi_n)= E_n=\#\{ (S_{n-1}\times S_1)-\rm{orbits\ on\ maximal\ chains\ of\ \Pi_n}\} .\end{equation}  

Thus, the phenomenon of having \textit{two} successive Euler numbers associated to $\Pi_n$  manifests itself in this context; it is the only context known to us other than \cite{Su1}.  
 It seems natural to ask:

\begin{qn} Does  this 
generalise to the rank $n$ root systems $B_n$ and $D_n$? What is the result of   considering orbits of  parabolic subgroups (e.g., the Weyl group of $B_{n-1}$) on the maximal chains of the  intersection lattices ?
\end{qn}

A useful recurrence relating the two families of numbers $b_S(n)$ and $b'_S(n)$ is as follows:

\begin{prop}(\cite{Su1}, Proposition 4.9, p. 284) Let $S$ be a subset of $[1, n-2]$ and suppose $1\notin S.$ Let $S-1$ denote the subset of 
$\{1,\ldots,n-3\}$ obtained by subtracting 1 from each element of $S.$ Then 
$$b_{S\cup\{1\}}(n)+b_S(n)= b_{S-1}'(n-1).$$
\end{prop}

Proposition 2.8 is derived 
from the  three-term recurrence of Theorem 2.2, which completely determines all the rank-selected homology representations.  This recurrence is  intrinsic to the recursive nature of the partition lattice; it arises from the fact that upper intervals are isomorphic to smaller partition lattices.   See \cite{Su1} for details.

\begin{qn}
  Is there  a natural topological explanation, again in terms of quotient complexes, for the recurrence of Proposition 2.8?
\end{qn}

Recall that  $[a,b]$ denotes the subset of consecutive ranks $\{a, a+1,\ldots, b\}.$
Partitionings of the quotient complex of $\Delta(\Pi_n)$ by $S_n$ and $S_{n-1}\times S_1$ were used successfully by Hanlon and Hersh to prove two conjectures of (\cite{Su1}, p. 289): 

\begin{thm} \cite{HH} (Theorems 2.1 and 2.2)  $b_S(n)\neq 0$ if $1\notin S$ or if $S=[1,r]\cup T$ where $T$ is a subset such that $\min T\geq r+2$ and $|T|\geq r.$
\end{thm}

When $S=[1,r],$ it was shown in (\cite{Su1}, Proposition 4.10 (2)) that $b'_S(n)=1,$ and it was conjectured that this uniquely characterises the first $r$ consecutive ranks  (\cite{Su1}, p. 289), and that $b'_S(n)$ is always nonzero.  This conjecture was also proved by Hanlon and Hersh. 

\begin{thm} \cite{HH} $b_S'(n)>1$ unless  $S=[1,r]$ for some $r$ 
with $2\leq r\leq n-1.$  
\end{thm}

We summarise the other known results about $b_n(S)$ and $b_n'(S)$ below.  Some partial results concerning the case of an arbitrary consecutive set of ranks are recorded in (\cite{Su1}, Theorem 4.8), but even this case seems to be difficult in general.

\begin{thm} $b_n(S)=0$ whenever any of the following conditions are met:

\begin{enumerate}
\item[(1)] (\cite{Ha1};\cite{St2}, Proposition 7.8) $S=[1,i],$ $1\leq i\leq n-2.$ 
\item[(2)] (\cite{Su1}, Theorem 4.2 and Remark 4.10.1) $S\supseteq [1, \lfloor\frac{n+1}{2}\rfloor]$
\item[(3)] (\cite{Su1}, Theorem 4.3 and Remark 4.10.1) $S=[1,r]\cup\{a\}$ for $a\notin [{r+2\choose 2}, n-r-1].$ 
\item[(4)] (\cite{HH}, Theorem 2.4);  \cite{Su1}, Theorem 4.7(2) and Remark 4.10.1) $S=[1,r]\backslash\{k\}$ for $k>r/2.$

\end{enumerate}
\end{thm}

\section{Permutation modules in homology}  

Stanley's paper focuses as much on the maximal chains in the partition lattice as on the homology.  In fact, the homology representation on the maximal chains has a very interesting structure, which we now describe.

As before, let $h_n$ be the homogeneous symmetric function of degree $n$.  Thus $h_n$ is the Frobenius characteristic of the trivial representation of $S_n.$ 
See \cite{M} and \cite{St4} for definitions.  Define positive integers $a_i(n)$ by the recurrence 
\begin{equation}
a_i (n+1)=i a_i(n)+ (n-2i+2) a_{i-1}(n)
\end{equation}
with initial conditions $a_0(1)=1=a_1(2), a_0(n)=0, n>1,$ and $a_i(n)=0$ if $2i>n.$   The $a_i(n)$ count the \textit{simsun} permutations in $S_{n-2}$ with $(i-1)$ descents; they were defined combinatorially  in \cite{Su1}. The recurrence (4) arose naturally from the recursive plethystic generating function for the action of $S_n$ on the maximal chains of $\Pi_n,$  which also led to the following more elegant  description of the $S_n$-action on the maximal chains:

\begin{thm} (\cite{Su1}, Theorem 3.2) The $S_n$-action on the  maximal chains of $\Pi_n$ decomposes into orbits as follows:
$$\alpha_n = \sum_{i=1}^{\lfloor\frac{n}{2}\rfloor} a_i(n) W_i,$$ where $W_i$ is the transitive permutation representation of $S_n$ acting on the cosets of the Young subgroup $S_2^i\times S_1^{n-2i}.$ 
In symmetric function terms, the Frobenius characteristic of  $\alpha_n$ is $$ \sum_{i=1}^{\lfloor\frac{n}{2}\rfloor} a_i(n) h_2^i h_1^{n-2i}.$$

\end{thm}

By extracting the multiplicity of the trivial representation, one obtains a  refinement of the Euler number $E_{n-1}$ via the simsun permutations.  The simsun permutations appear extensively in the literature.  See \cite{Su1} for the original definition, and also \cite{St5} and \cite{St6}.  (Note that the recurrence given in (\cite{St5}, Theorem 3.3) for 
the number $f_k(n)$ of simsun permutations in $S_n$ with $k$ descents is obtained from equation (4) above by means of the substitution $a_i(n)=f_{i-1}(n-2).$ )

Now let $m=2n$ be even, and consider the subposet $\Pi_{2n}^e$ of $\Pi_{2n}$ consisting of partitions with an even number of blocks.  This is the rank-selected poset corresponding to selecting alternate ranks in $\Pi_{2n},$ beginning with rank 2.  Let $R_{2n}$ 
 denote the Frobenius characteristic (see \cite{M}, \cite{St4}]).  
Define integers $\{b_i(n),\ 2\leq i\leq n\},$  
 by means of the recurrence 
 $b_2(n)=1,\ n\geq 2,$ ($b_i(n)=0$ unless $2\leq i\leq n$), 
$$b_i(n)=\sum_{k\geq 0}{2n-2i+k \choose k} 
\sum_{r\geq 1} (-1)^{r-1} {i-k\choose i-2r} b_{i-k}(n-r).$$

The main result of \cite{Su2} states that:

\begin{thm}(\cite{Su2}, Theorem 2.5) $$R_{2n}=\sum_{i=2}^n b_i(n) h_2^i h_1^{2n-2i},$$  
\noindent and thus the character values are nonzero only on the involutions of $S_{2n}.$ 
\end{thm}

  It was conjectured in \cite{Su2} that the integers $b_i(n)$ are nonnegative; this result was proved by Benjamin Joseph (a student of Stanley) in his thesis.  By analysing the recurrence for the $b_i(n)$ within the framework of a sign-reversing involution, Joseph devises an intricate and ingenious  set of objects enumerated by the $b_i(n).$ 
The objects are cleverly constructed  to be the fixed points of the sign-reversing  involution.

\begin{thm} (\cite{J}, Chapter 4, Section 3, Theorem 8) $b_i(n) $ is a positive integer for all $n\geq i\geq 2.$ 

\end{thm}

Hence, putting together these two results, since $h_2^i h_1^{2n-2i}$ is the Frobenius characteristic of the transitive permutation module induced from the Young subgroup $S_2^i\times S_1^{2n-2i}$, another conjecture  (\cite{Su2}, Conjecture 2.7) is established:

\begin{cor}  The top homology of $\Pi_{2n}^e$ is  a permutation module for $S_{2n}$ whose character values are supported on the set of involutions.
\end{cor}
Note that by Theorem 3.1, the identical statement is true (mutatis mutandis) for the action of $S_n$ on the maximal chains  of  $\Pi_n.$

\begin{qn}
 Corollary 3.4 is still a somewhat mysterious fact; is there   a more natural topological explanation? 
\end{qn}

Define integers $E_k(n)$ by $$E_k(n)=\sum_{i=2}^n b_i(n) { n-i\choose k-i},\ \  2\leq k\leq n.$$ 
It follows from Joseph's work that the $E_k(n)$ are also nonnegative integers (see \cite{Su2}, Conjecture 3.1).  Their representation-theoretic significance lies in the fact that 
the Frobenius characteristic of the homology of $\Pi_{2n}^e$ can be rewritten as (see \cite{Su2}, Corollary 2.8) 
$$R_{2n}=\sum_{i=2}^n E_i(n)  h_2^i e_2^{n-i}.$$
In this formulation, $E_n(n)$ is the multiplicity of the trivial representation in $R_{2n}.$  But $R_{2n}$ is a submodule of $\alpha_{2n},$ and  the number of $S_{2n}$-orbits in the latter module is $E_{2n-1}$ by Stanley's result.  Thus $E_n(n)$ must count a subset of the alternating permutations enumerated by $E_{2n-1}.$ 

\begin{qn}
Is there a combinatorial description of the integers $E_n(n)$ as a subset of alternating permutations in $S_{2n-1}$?  Is there one that would lead to a similar description for the integers $b_i(n)$? 
\end{qn}

Finally, let $\Pi_{2n}^e(\bar{k})$ denote the subposet of $\Pi_{2n}^e$ obtained by selecting the top $k$ nontrivial ranks.   
 We have one more 
conjecture which generalises Theorem 3.3:

\begin{conj} (\cite{Su2}, Conjecture 2.10) For $1\leq k\leq n-1,$ $S_{2n}$ acts on the homology of $\Pi_{2n}^e(\bar{k})$ as a permutation module 
which can be written as a sum of induced modules, in which the stabilisers are all Young subgroups of $S_{2n}.$  In terms of symmetric functions, the Frobenius characteristic of the homology is 
$h$-positive.

\end{conj}

In \cite{HH},  Hanlon and Hersh describe a partitioning for the quotient complex of 
$\Delta(\Pi_n)/S_\lambda$ for an arbitrary Young subgroup $S_\lambda.$ 

\begin{qn} Can such a partitioning be adapted to prove the above conjecture, using the  sequence of poset inclusions  
$$\Pi_{2n}^e(\bar{1})\subset \Pi_{2n}^e(\bar{2}) \ldots \subset  \Pi_{2n}^e(\overline{n-1})=\Pi_{2n}^e?$$ 
\end{qn}

Recall from Theorem 3.1 that the Frobenius characteristic of $\alpha_{2n}$ is also a nonnegative integer  combination of the homogeneous symmetric functions $h_2^i h_1^{2n-2i}.$    By Stanley's observation (2.1), $R_{2n}$ is (the Frobenius characteristic of) a submodule of $\alpha_{2n}.$   The data supports  a stronger statement:

\begin{conj} $\alpha_{2n}-R_{2n}$ is a nonnegative integer combination of
  homogeneous symmetric functions $h_\lambda$ indexed by integer partitions $\lambda$ of $n$ with parts equal to 1 or 2.  This is equivalent to the enumerative statement that 
\begin{equation} b_i(n)\leq a_i(2n) {\ \rm for\ all\ } 2\leq i\leq n.
\end{equation}

\end{conj}

The data for $n\leq 7$ overwhelmingly  verifies the above inequality.  In view of this, one may ask:


 \begin{qn}
Is there a way to explain Theorem 3.2 by isolating specific  chains in $\Pi_{2n}$?  
\end{qn}

\begin{qn}
Is there an injection from Joseph's set of objects counted by the $b_i(n)$ to the simsun permutations in $S_{2n-2}$ with $(i-1)$ descents, or a different set of objects counted by $a_i(2n)$?  
\end{qn}

\begin{qn}Does the topology of the quotient complex $\Delta(\Pi^e_{2n})/S_{2n}$ offer any insight in this special case?  
\end{qn}

For many other  enumerative and completely elementary conjectures (still open, as far as we know) regarding refinements of the simsun numbers $a_i(2n)$ and the Genocchi numbers, see \cite{Su2}.  Tables of values for the $b_i(n)$ appear in \cite{Su3}.  For values of the $a_i(m),\ m\leq 14,$ see (\cite{OEIS}, sequence No. A113897).

\section{Partitions with forbidden blocks: pure and non-shellable}

A central theme of Stanley's paper \cite{St2} is to take a representation-theoretic result and extract interesting enumerative identities from it.  In \cite{Su4} and \cite{Su5},  representation-theoretic results were used to predict the topology of the order complex. 

Our study of the posets described in this section begins with a 
Lefschetz module calculation (see \cite{St2} for definitions), and the discovery that the resulting Frobenius characteristic has interesting properties.  The results of the preceding sections were obtained by using two tools introduced in \cite{Su1} which have proved to be very powerful. The first exploits the acyclicity of Whitney homology (\cite{Su1}, Lemma 1.1).  
Coupled with the innate plethystic nature of the partition lattice, this technique allows one to write down, with relative ease,  generating functions for the Lefschetz module of any subposet of partitions with arbitrary block sizes (\cite{SuJer}).  The considerable machinery of symmetric functions can then be used to analyse the representations, and thereby predict topological properties.  (As a simple example, if the Lefschetz module is {\it not} plus or minus a true module, one concludes immediately that there is homology in more than one degree.) These ideas were applied  successfully in 
\cite{SWa} and \cite{SWe} to the $k$-equal lattice and the $k$-equal  subspace arrangement \cite{BWe}.  The  $k$-equal lattice  served as the motivating example for the theory of nonpure shellability developed by Bj\"orner and Wachs \cite{BW}; their shelling was a crucial ingredient in isolating the homology representations by degree.   The Whitney homology   technique was subsequently generalised  by Wachs to obtain far-reaching and beautiful  results  (see \cite{Wa} for many applications), and combined with shellings and other constructive methods, to refine the resulting plethystic identities for homology by degree.
   
 A second technique in \cite{Su1} proved particularly useful in computing homology  in the case when an antichain is deleted from a poset (rank-selection falls into this category 
by repeated deletion of ranks).  It may be viewed as  a group-equivariant  homology version of (\cite{Ba2}, Lemma 4.6).
The general group-equivariant result appears in (\cite{Su1}, Theorem 1.10 and Remark 1.10.1), while the result for antichains appears explicitly  in 
(\cite{Su4}, Theorem 4.2).   In \cite{Su4} the antichain result was also derived in a more general context using the long exact homology sequence of a pair.  The key step is to determine the 
relative homology of the pair, and can also be obtained from the Homotopy Complementation Formula   
of Bj\"orner and Walker \cite{BWal}. The homological antichain result is crucial to  the  Lefschetz module computations described here, as well as the three-term recurrence (Theorem 2.2) of the preceding section, and once again can be  effectively combined with 
the recursive and plethystic nature of the partition lattice. 

  Recall that the modular elements of the partition lattice (\cite{St3}, Example 3.13.4) are precisely the partitions with a unique non-singleton block.  For $k\geq 2$ let 
$Q_n^k$ be the subposet of $\Pi_n$ consisting of all partitions except 
those with $(n-k)$ blocks of size 1 and one block of size $k,$  and let $P_n^k=\cap_{i=2}^k Q_n^i.$  Thus $Q_n^k$ is obtained by deleting from $\Pi_n$ all the modular elements in which the singleton block has size $k,$ while $P_n^k$ is obtained by removing all modular elements where the singleton block has size $i,$ for $2\leq i\leq k.$   These posets were studied in \cite{Su4}. For $k\geq 3,$ $Q_n^k$ is pure of full rank $n-1,$ but $P_n^k$ is pure of rank $n-2,$ since all the atoms have been deleted.

  Let $\pi_n$ denote the Frobenius characteristic of the $S_n$-module afforded by the homology of $\Pi_n,$ and define the symmetric function $\pi_{n,k}$ by 
$$\pi_{n,k}=\pi_k h_1^{n-k} -\pi_n.$$ 
Using the preceding methods, it was shown  (\cite{Su4}, Theorem 4.2 and Theorem 4.4) that the Lefschetz module  of the order complex of $Q_n^k$ has Frobenius characteristic 
$(-1)^{n-4} \pi_{n,k}.$ 
It is easy to see that $\pi_{n,k}$ is in fact a true representation of $S_n.$  Even more interestingly, the restriction of $\pi_{n,n-1}$ to $S_{n-1}$ is $\pi_{n-1}.$ Hence (with unabashedly wishful thinking)\footnote{ cf. Richard Stanley's talk at his 70th birthday conference.} the representation theory might be said to predict that $Q_n^k$ has homology only in degree $(n-4),$ and that $Q_n^{n-1}$ and $\Pi_{n-1}$ have the same homology, or even the same homotopy type. Miraculously, this is indeed the case.
 
\begin{thm}  (See (\cite{Su4}, Theorem 2.1, Theorem 3.5, Theorems 4.5 - 4.6.)
\begin{enumerate} 
\item   The inclusion map between the $(n-4)$-dimensional order complex of  $P_n^k$ and the $(n-3)$-dimensional order complex of $Q_n^k$ is an $S_n$-equivariant homotopy equivalence. 
\item  There is a  map from  the poset of nonmodular partitions $P_n^{n-1}$ to $\Pi_{n}$  whose image is $\Pi_{n-1},$ and which induces an $(S_{n-1}\times S_1)$-equivariant homotopy  equivalence.
\item   $P_n^k$ is Cohen-Macaulay of rank $(n-2)$, $Q_n^k$ (of rank $n-1$) is not.  The pure poset $Q_n^k$ has unique nonvanishing homology in degree $(n-4),$ one less than the top degree.
\item The action of the symmetric group $S_n$ on the unique nonvanishing homology of each of the posets $P_n^k$ and $Q_n^k$ has Frobenius characteristic 
$$\pi_k h_1^{n-k} -\pi_n.$$

\end{enumerate}
\end{thm}

The representation described by $\pi_{n,k}$ is what was called the generalised Whitehouse module in \cite{Su4}; in the special case $k=n-1,$ it is the Whitehouse module (also named thus in \cite{Su4})  or tree representation of  \cite{RW},  whose restriction to $S_{n-1}$ is  $\pi_{n-1},$ the representation of $S_{n-1}$ on the top homology of the partition lattice $\Pi_{n-1}.$ The generalised Whitehouse module shows up again in \cite{Su5}.  

It was also shown that 

\begin{thm} (\cite{Su4}, Theorem 2.12) For  $k\geq 2,$ the order complex of $Q_n^k$ has the homotopy type of a wedge of spheres of dimension $(n-4).$ 
\end{thm}

The proof  proceeds somewhat indirectly, by establishing that the order complex is simply-connected using a technical lemma of  Bouc (\cite{Bo} (9, Section 2.2.2, Lemme 6).  Bouc's lemma is misstated in his original paper.  The proof in \cite{Su4} therefore requires a slight modification which appears in (\cite{Su5}, p. 276).

\begin{qn}
 Is there a more direct way to prove Theorem 4.2, i.e., to establish the homotopy type of the posets $Q_n^k$ or $P_n^k$?
\end{qn}

The  restriction of $\pi_{n,k}$ to $S_{n-1}\times S_1$ has Frobenius characteristic equal to   $(n-k)\pi_kh_1^{n-k-1}.$ This leads to 

\begin{qn} (More wishful thinking?) There is an $S_{k}\times (S_{n-k-1}\times S_1)$-module isomorphism   between the homology of the poset $Q_n^k$ 
(or $P_n^k$), and $(n-k)$ copies of the homology of $\Pi_k\times \Pi_{n-k}.$  (Note that  dimensions agree for the order complexes if we consider $P_n^k:$  $(k-3)+(n-k-3)+2=n-4$ .) Can this be explained topologically via a simplicial (poset) map, as it was in Theorem 4.1 for the case $k=n-1$?
\end{qn}

It follows from Theorem 4.1 that the relative homology group of the pair $(\Pi_n, P_n^k)$ vanishes in degrees different from $n-3,$ and, in degree $(n-3),$ equals the direct sum of $\tilde{H}_{n-3}(\Pi_n)$ and $\tilde{H}_{n-4}(P_n^k).$ Hence the relative homology group 
$\tilde{H}_{n-3}(\Pi_n, P_n^k)$ affords the representation 
$\pi_k h_1^{n-k}.$  Note that the relative chain complex of the pair consists of chains passing through a modular element 
with at least $(n-k)$ singletons.

\begin{qn} Does the quotient 
complex $\Delta(\Pi_n)/\Delta(P_n^k)$ have a nice topological description that is consistent with the representation it affords? 
(When $P_n^k$ is replaced by $Q_n^k,$ the answer is provided by 
the Homotopy Complementation Formula of \cite{BWal}.) 
\end{qn}

The posets $Q_n^k$ appear to be related to two other classes of subposets, 
which we now describe.
First let $\Pi_{n, \leq k}$ denote the subposet of $\Pi_n$ consisting of 
partitions with block size at most $k.$  Note that one has natural inclusions 
\begin{equation}
\Pi_{n,\leq 2}\subset \ldots \subset \Pi_{n,\leq k} \subset \Pi_{n,\leq k+1}\ldots \subset \Pi_{n,\leq n-2},
\end{equation}
and that $\Pi_{n,\leq n-2}$ is precisely the poset $Q_n^{n-1}.$ By 
(1) and (2) of Theorem 4.1, this last order complex has the same $S_{n-1}$-homotopy type as that of $\Pi_{n-1},$ and its $S_n$ homology representation is given by the Whitehouse module.

 For $k=2,$ this is the matching complex studied by many authors in different guises (see \cite{Wa} for an extensive list).  Homology and homotopy type were determined by Bouc \cite{Bo}, who also showed that there is torsion.   For arbitrary $k$ these posets were first considered in \cite{BL}, where their M\"obius function was computed; they   
 also arise in certain examples of relative arrangements,
a concept introduced in \cite{Wel}.  
 Welker conjectured that the integral homology of these 
posets is free.  This  was established for certain values of $n:$

\begin{thm} (\cite{Su5}, Theorem 2.8, Theorem 2.10, Corollary 2.15) Let  $k\geq 3.$
\begin{enumerate}
\item For fixed $n,$ when $n<2k+2,$ all the posets $\Pi_{n, \leq k}$  have the same $S_n$-homotopy type as the posets $Q_n^{n-1};$  hence 
the $S_n$-representation on the homology is the Whitehouse module 
(Frobenius characteristic $\pi_{n-1} h_1-\pi_n$).
\item When $n=2k+2,$ the $(2k-2)$-dimensional order complex of $\Pi_{2k+2, \leq k}$ is homotopy equivalent to a wedge of 
$(2k-3)$-dimensional spheres.
\item When $n=3k+2,$ the $(3k-3)$-dimensional order complex of $\Pi_{3k+2, \leq k}$ is homotopy equivalent to a wedge of 
$(3k-4)$-dimensional spheres.
\end{enumerate}
\end{thm}

The case $n=2k+2$ deserves special mention.  The homotopy equivalence was again proved indirectly.  Interestingly, 
there is an $S_{2k+1}$-isomorphism in homology with the order complex of $Q_{2k+1}^{k+1}.$ 

\begin{qn}
Is there  a natural $S_{2k+1}$-homotopy equivalence between (the order complexes of) 
$\Pi_{2k+2, \leq k}$ and $Q_{2k+1}^{k+1}$?  What, if anything,  does restricting block sizes have to do with removing modular elements?
\end{qn}

Now let $\Pi_{n, \neq k}$ denote 
the poset of partitions with no block of size $k.$   (Again the M\"obius function  for these posets was computed in \cite{BL}.) 
Observe that $\Pi_{n, \neq 2}$ is the 3-equal lattice of \cite{BWe}.   Bj\"orner and Welker determined the homotopy type of the  $k$-equal lattice to be a wedge of spheres of varying dimensions, while the  homology representations were determined in \cite{SWa}.  
Our original motivation for the study of both $\Pi_{n, \neq k}$ and $\Pi_{n, \leq k}$ arose from Lefschetz module calculations appearing in \cite{SuJer}. 

\begin{thm} (\cite{Su5}, Theorem 4.3, Theorem 4.8) Assume that $k\geq 3.$
 For $n\leq 2k$ the posets  $\Pi_{n, \neq k}$
 and the poset $Q_n^k$ have homotopy equivalent 
order complexes; when $n<2k$ they have  isomorphic $S_n$-homology modules, namely, the generalised Whitehouse module described in Theorem 4.1 (4) above.

\end{thm}


\begin{conj}
The $(2k-2)$-dimensional  order complex of $\Pi_{2k+1, \neq k}$ has the homotopy type of a wedge of spheres of dimension $2k-3.$ 
\end{conj}

It is shown in \cite{Su5}  that the integral homology can only occur in degrees $2k-3$ and $2k-4.$ The M\"obius number of the poset is computed in (\cite{Su5}, Theorem 4.5).

The posets described in this section have the property that they are  pure with  non-vanishing homology only in one less than the top dimension, and also have the homotopy type of a wedge of spheres.  
These posets are clearly not shellable, so the methods of \cite{Bj1} and  \cite{BW} do not apply.   

\begin{qn} Is there a general edge labelling that would detect this phenomenon of collapse in homology and homotopy type from the top dimension to a lower one, and  corroborate the homotopy type of these posets?
\end{qn}

\section{Stability in Homology}

This section is prompted by another observation in Stanley's paper  \cite{St2}:
For a fixed subset $S$ of ranks, the numbers $b_S(n)$ stabilise at $n=2\max(S).$ That is,  $b_S(m)= b_S (2(\max S))$ for all $m\geq 2\max S.$  In particular, if $b_S(2\max(S))=0,$ then $b_S(n)=0$ for all 
$n\geq 2(\max S).$  

In view of equations (2) and (3) of Section 2, Stanley's statement is equivalent to the analogous statement for the multiplicities $a_S(n)$ of the trivial representation in the action of $S_n$ on the maximal chains of the rank-selected subposet $\Pi_n(S).$   Stability in the special case of the trivial representation is \lq\lq not difficult to see" (\cite{St2}); this will be evident in the proof of Theorem 5.3 below, where we establish a somewhat more  general result.
   Stanley's stability result for $b_S(n)$ was also proved using  Hersh's partitioning of $\Delta(\Pi_m)/S_m$   (\cite{HH}, Theorem 2.5). 

Coincidentally, there has been a recent surge of interest in such stability questions, in the work of Church and Farb as well as others.  See \cite{CF} and the references therein.

To avoid confusion and trivialities, we point out that 
$a_\emptyset(n)=1=b_\emptyset(n)$ and $a_{\{1\}}(n)=1$ for all $n\geq 2.$

\begin{prop}  (\cite{St2}, p. 152) Let $S$ be any (nonempty) subset of the nontrivial ranks $[1,n-2].$   
\begin{enumerate}
\item  (implicit in \cite{St2}) $a_S(n)= a_S (2\max S)=\overline{a}_S, \  n\geq 2(\max S),$   and hence
\item 
$b_S(n)= b_S (2(\max S)=\overline{b}_S$ for all $n\geq 2(\max S).$
\end{enumerate}

\end{prop}

It turns out that this stability is transferred to the multiplicities $b'_S(n)$ as well.  Although this is also a special case of Theorem 5.3 below, it is worth mentioning separately because in this case  the formulas derived  in \cite{Su1} allow us to write down relatively simple and interesting relationships between the stable values.   The stability also transfers to the numbers  $a'_S(n)$ which denote the number of $( S_{n-1}\times S_1) $-orbits in the $S_n$-action on the rank-selected subposet corresponding to the set $S.$   For a fixed set $S,$ we write $\overline{a}_S$ for the limiting value of $a_S(n)$ as $n$ approaches $\infty,$ and similarly $\overline{b}_S=\lim_{n\rightarrow\infty} b_S(n),$ 
 $\overline{b'}_S=\lim_{n\rightarrow\infty} b'_S(n) $  and 
 $\overline{a'}_S=\lim_{n\rightarrow\infty} a'_S(n).  $  We record  the actual expressions relating the limiting values in:

\begin{prop}  Fix $n$ and a nonempty subset $S\subseteq [1,n-2]$ such that $n\geq 2 \max(S)+1.$ Then 
\begin{enumerate} 
\item $a'_S(n)=a'_S(2\max S +1)=\overline{a'}_S$ for all $n\geq 2(\max S)+1,$ and 
 $\overline{a'}_S =\overline{a}_{(S+1)\cup\{1\}}.$
\item $b'_T(n)=b'_T(2\max T +1)=\overline{b'}_T$ for all $n\geq 2(\max T)+1, $  and $\overline{b'}_T=\overline{b}_{\{1\}\cup (T+1)}+\overline{b}_{(T+1)}.$
\item The multiplicity of the irreducible indexed by $(n-1,1)$ (the reflection representation of $S_n$) also stabilises in both $\alpha_S(n)$ and $\beta_S(n)$ when  $n\geq 2\max(S)+1.$  The stable values are 
  $\overline{a}_{\{1\}\cup (S+1)}-\overline{a}_S$ 
and $\overline{b}_{\{1\}\cup (S+1)}+\overline{b}_{(S+1)}-
\overline{b}_{S}$ respectively.
\end{enumerate}
\end{prop}

\begin{proof} 
Part (1) can be proved using the following special case of the plethystic recurrence (\cite{Su1}, Proposition 3.1; see also Theorem 2.2 of this paper) which was crucial to the construction of the simsun permutations of Rodica Simion and this author (see Theorem 3.1).
Let 
$S=\{s_1<\ldots<s_k\}$ be a nonempty subset not containing 1, and let $S-1$ denote the subset obtained by subtracting 1 from each element of $S.$ Then, in terms of symmetric functions, 
\begin{equation}
\alpha_{\{1\}\cup S}(n)=h_2\frac{\partial}{\partial p_1} \alpha_{S-1}(n-1). 
\end{equation}

The multiplicity of the trivial representation on the left-hand side is $a_{\{1\}\cup S}(n).$   On the right-hand side,
 it is $$\langle \frac{\partial}{\partial p_1} \alpha_{S-1}(n-1), h_{n-2}\rangle,$$
which equals $ a'_{S-1}(n-1).$ Thus we have 
$$ a_{\{1\}\cup S}(n)= a'_{S-1}(n-1)$$
 
Now replace $S-1$ by $S,$  and $n-1$ by $n.$ 
Then 
\begin{equation} a_{\{1\}\cup (S+1)}(n+1)= a'_{S}(n)
\end{equation} 
The left-hand side equals its limiting value $\overline{a}_{\{1\}\cup (S+1)}$ when $n+1\geq 2\max(S+1),$ i.e., when $n\geq 2\max S +1$ as stated in Part (1).

Note that while this also establishes the stability of $b'_S(n)$ for $n$ sufficiently large, in order to compute the actual stable value, it is more efficient to  invoke  the recurrence of Proposition 2.8. 
Writing $T$ for the set $S-1$ and $m$ for $n-1$ in that recurrence, we 
have 
\begin{equation}
b'_T(m)=b_{(T+1)\cup\{1\})}(m+1)+b_{(T+1)}(m+1)
\end{equation}
which in turn equals 
                 $\overline{b}_{(T+1)\cup\{1\}}+\overline{b}_{(T+1)},$ provided $m+1\geq 2\max(T+1)=2(\max T) + 2, $ i.e., provided 
$m\geq 2(\max T)+1$ as desired. (Here $T+1$ is the subset obtained from $T$ by adding 1 to each element of $T$.)

For Part (3) it suffices to note that the multiplicity of the reflection representation is the difference $a'_S(n)- a_S(n)$ (respectively 
$b'_S(n)-b_S(n)$). 
\end{proof}

By analysing the $S_n$-orbits of maximal chains in $\Pi_n(S),$ we are able to derive a slightly stronger stability result.  The reader is referred to \cite{M} for the following definitions.  Recall  that $s_\lambda$ denotes the Schur function indexed by the integer partition $\lambda;$ it is the Frobenius characteristic of the irreducible indexed by $\lambda.$  Also recall the operation $D_\mu$ of {\it skewing} a (possibly virtual) representation $f$ of $S_n$ by the irreducible indexed by a partition $\mu$ of $m:$  its image is the representation $D_\mu f $  of $S_{n-m}$ defined by 
$\langle D_\mu f, s_\nu\rangle=\langle f, s_\mu s_\nu\rangle$  for all partitions $\nu$ of $n-m.$ Here $\langle,\rangle$ denotes the inner product of $S_n$-characters (or Frobenius characteristics) which makes the Schur function an orthonormal basis in the ring of symmetric functions. 

\begin{thm} Fix  a nonempty subset $S$ of the nontrivial ranks $[1,n-2]$ and an integer  $k\geq 0$.  Consider the Schur functions appearing in the linear span of the set $U_k=\{h_{n-\ell} s_\mu: 0\leq \ell\leq k, \mu\vdash \ell\}.$ Then for $n\geq 2\max(S) +k,$ the 
multiplicities of the irreducible indexed by $\lambda$ in  $\alpha_S(n)$ and 
$\beta_S(n)$ stabilise for every partition 
$\lambda$ of $n$ such that $s_\lambda$ is in $U_k.$  Equivalently, 
for 
$n\geq 2\max(S) +k,$
the operation of skewing the representations  $\alpha_S(n)$ and 
$\beta_S(n)$ by the irreducible indexed by $(n-\ell)$ results in stable modules for all $0\leq \ell \leq k.$

\end{thm}

\begin{proof} Let $\nu$ be any integer partition of $k,$ and denote by $S_\nu$ the Young subgroup determined by $\nu.$ The argument we give here  will show  that the number of $(S_{n-k}\times S_k)$-orbits, and more generally that the number of $S_{n-k}\times S_\nu$-orbits, in $\alpha_S(n)$ stabilises; in the case $k=0$  it must reduce to Stanley's original argument (unstated in \cite{St2}).  Let $S=\{s_1<s_2<\ldots<s_r\},$ and assume $n\geq 2s_r +k.$ The partitions at rank $s_r$ in $\Pi_n$ have $n-s_r$ blocks, and their {\it type}, (defined as the integer partition of $n$ whose parts encode the block sizes) must be an integer partition of $n$ of length $n-s_r,$ of the form $$\mu_1+1,\mu_2+1,\ldots, \mu_\ell+1, 1,1,\ldots, 1$$  where $\mu=(\mu_1\geq\mu_2\geq\ldots \geq \mu_\ell)$ is an integer partition of $s_r.$  The fact that $n\geq 2 s_r+k$ implies that
$s_r\leq n-s_r-k\leq n-s_r,$ ensuring that the set partitions $\tau\in\Pi_n$ at rank $s_r$ are in one-to-one correspondence with the integer partitions $\mu$ of $s_r.$  Furthermore, it is clear that there must be at least $n-s_r-|\mu|=n-2 s_r$ parts of size 1, i.e., each set partition 
has at least $n-2s_r$ singleton blocks.  But $n-2s_r\geq k,$ so this means every set partition at rank $s_r$ has at least $k$ singleton blocks.    Each such partition determines a distinct $(S_{n-k}\times S_\nu)$-orbit of the  chains supported by the rank set $S.$  Since the remaining elements of the chain lie below the 
elements at rank $s_r,$ they depend only on the choice of the integer partition $\mu$ of $s_r=\max S,$ and not on $n.$  Hence the number of $(S_{n-k}\times S_\nu)$-orbits in $\alpha_S(n)$ is constant for such values of $n$. Since the homogeneous symmetric function $h_\nu$ is the Frobenius characteristic of the trivial representation of $S_\nu$, we conclude that 
$\langle\alpha_S(n), h_{n-k} h_\nu\rangle$ is constant for $n\geq 2\max(S) +k. $  

Now use the fact \cite{M} that the functions $h_\nu$ form a 
basis for the symmetric functions of degree $k,$ and in fact each Schur function $s_\nu$ is an integer combination of the $\{h_\lambda\},$ as $\lambda$ runs over all integer partitions of $k.$
We conclude that the inner product $\langle\alpha_S(n), h_{n-k} s_\nu\rangle$ is stable for all partitions $\nu$ of $ k$ and  $n\geq 2\max(S)+k.$
The result follows with the final observation that if $n\geq 2\max(S)+k,$ then  necessarily $n\geq 2\max(S)+\ell$ for all $\ell$ between 0 and $k.$ 
\end{proof}

\begin{remark}
  Restated in the language of \cite{CF}, this theorem shows that, for fixed $k$ and nonempty $S,$  the sequences of (skew) representations $\{D_{(n-k)} \alpha_S(n)\}_n$ and 
$\{D_{(n-k)} \beta_S(n)\}_n$ are {\it representation stable}, whereas 
the sequences $\{\alpha_S(n)\}_n$ and $\{\beta_S(n)\}_n$ are $\lambda$-{\it representation stable} for  some choices (see Corollaries 5.6 and 5.8 below) of $\lambda,$ e.g., $ (n-k,k)$ and $ (n-k, 1^k).$
\end{remark}

\begin{qn} Can this stability of the action on the chains (simplices) be explained via the topology of the quotient complex  $\Delta(\Pi_n(S))/(S_{n-k}\times S_\nu)$ for an integer partition $\nu$ of $k$?  
Note that this would involve looking only  at  flag $f$-vectors, rather than  $h$-vectors as in \cite{HH}.
\end{qn}


Theorem 5.3 allows us to establish stability for more irreducibles; we single out   the cases of two-part partitions and hooks:

\begin{cor}  Let $V(k)$ denote the $S_n$-irreducible indexed by the integer partition $(n-k,k),$  where $n-k\geq k\geq 0.$ Then the multiplicity of $V(k)$ in $\alpha_S(n)$ and hence in $\beta_S(n)$ stabilises for $n\geq 2\max(S)+k.$
\end{cor}

\begin{proof} Applying Theorem 5.3 with $\nu=(k)$ yields the stability of the inner products $\langle\alpha_S(n), h_{n-\ell} h_\ell\rangle$ 
for $0\leq \ell \leq k.$ Since the Frobenius characteristic of   $V(k)$ is
$h_{n-k}h_k - h_{n-k+1} h_{k-1},$  the result follows.  
\end{proof}

\begin{remark}
The stable values of the preceding multiplicities can be determined as before.  For instance, we have that the stable multiplicity of $(n-2,2)$ in $\alpha_S(n)$ is (for $n\geq 2\max(S) +2$) $a_{(S+2)\cup 2} -2 a'_S.$
\end{remark}

\begin{cor} Let $V(1^k)$ denote the irreducible indexed by the hook shape $(n-k, 1^k)$ (with $k$ parts equal to 1).  Then the multiplicity of $V(1^k)$   in $\alpha_S(n)$ and hence in $\beta_S(n)$ is stable for $n\geq 2\max(S)+k.$
\end{cor}

\begin{proof} Take $\nu=(1^k)$ in Theorem 5.5.  Then $s_\nu=e_k,$ 
the $k$th elementary symmetric function.  The set $U_k$ in the statement  of Theorem 5.5  then contains the set $\{ (-1)^{k-i} h_{n-i} e_i\}_{i=o}^k.$  
But a well-known symmetric function identity \cite{M}  says that the Schur function corresponding to the hook $(n-k, 1^k)$ equals the alternating sum 
$$ \sum_{i=0}^k (-1)^{k-i} h_{n-i} e_i.$$ 
The result now follows as before.
\end{proof}

For definitive and far-reaching results on the stability question  in the rank-selected homology of $\Pi_n,$ see the forthcoming paper  \cite{HR}.

Finally, we have:

\begin{remark} Stability of another sort is exhibited by the homology of the order complexes (discussed in Section 4) of the sequences of posets $\Pi_{n, \leq k}$ and $\Pi_{n, \neq k},$ for $k>\frac{n-2}{2}$ and $k>\frac{n}{2}$ respectively (see  the chain of inclusions in equation (8)).  The significance (if any) of the half-way mark  is unclear.
\end{remark}

The families of subposets described in this paper often have natural filtrations associated with them. The topological results in \cite{Su4} and \cite{Su5} were obtained  by first establishing that the order complexes are simply connected, and then showing that homology is concentrated in a single degree. 
 There has  been an explosion in the  development of tools in geometric and topological combinatorics in recent years: see \cite{Wa}.  Perhaps these  techniques  can be used successfully on the problems 
discussed here.


\begin{thebibliography}{<99>}

\bibitem{Ba1}
   K. Baclawski,  \textit{Cohen-Macaulay ordered sets},
J. Algebra \textbf{63} (1980),  226--258.

\bibitem{Ba2}
   K. Baclawski,  \textit{Cohen-Macaulay connectivity and geometric lattices}, European J. Comb. \textbf{3} (1982),  293--305.

\bibitem{Bj1}  A. Bj{\"o}rner, 
\textit{Shellable and Cohen-Macaulay partially 
ordered sets}, Trans. Amer. Math. Soc. \textbf{260} (1980), 159--183.

\bibitem{Bj2}   \bysame, \textit{Topological Methods}, in 
\textit{Handbook of Combinatorics}, (R. Graham, M. Gr\"otschel and 
L. L\'ovasz, eds.), North-Holland (1995), 1819--1872.

\bibitem{BL} A. Bj\"orner and L. Lov\'asz, \textit{Linear decision 
trees, subspace arrangements and M\"obius functions},
J. Amer. Math. Soc. \textbf{7} (1994), 677--706.


\bibitem{BW} A. Bj\"orner and M. L. Wachs, \textit{Shellable nonpure 
 complexes and posets} I, Trans. Amer. Math. Soc. \textbf{348} (4) (1996), 1299--1327.

\bibitem{BWal} A. Bj\"orner and J. Walker, \textit{A Homotopy Complementation Formula for Partially Ordered Sets}, European J. Comb. \textbf{4} (1983), 11--19.

\bibitem{BWe} A. Bj\"orner and V. Welker, \textit{
The homology of \lq\lq $k$-equal \rq\rq manifolds and related partition 
lattices}, Adv. in Math. \textbf{110} (1995), 277--313.

\bibitem{Bo}  S. Bouc, \textit{Homologie de certains ensembles de 
2-sous-groupes du groupe sym\'etrique}, J. Algebra \textbf{150} (1992), 158--186.

\bibitem{CHR} A.R. Calderbank, P. Hanlon and R.W Robinson, 
 \textit{Partitions into even and odd block size and some unusual
 characters of the symmetric groups}, Proc. London Math. Soc. 
\textbf{3 (53)}, (1986), 288--320.

\bibitem{CF} T. Church and B. Farb, \textit{Representation theory and homological stability}, Advances in Mathematics \textbf{245} (2013), 250-314.


\bibitem{Ha1}  P.~Hanlon,
\textit{The fixed-point partition lattices}, Pacific J. Math. \textbf{96} (1981), 
319--341.

\bibitem{Ha2}  \bysame, \textit{A proof of a conjecure of Stanley concerning partitions of a set}, European J. Combin. \textbf{4} (1983), 137--141.

\bibitem{He}  P.~Hersh, \textit{Lexicographic shellability for balanced complexes}, J. Algebraic Combinatorics \textbf{17} (2003), No. 3, 225--254. 


\bibitem{HH} P.~Hanlon and P.~Hersh, \textit{Multiplicity of the trivial representation in rank-selected homology of the partition lattice}, 
J. Algebra, \textbf{266} (2003), No. 2, 521--538. 

\bibitem{HR} P. L. Hersh and V. Reiner, \textit{$S_n$-representation stability for Whitney homology and rank-selected homology of the partition lattice}, to appear.

\bibitem{J} B.~ Joseph, \textit{The Involution Principle and h-positive Symmetric Functions}, Thesis, Massachusetts Institute of Technology, September 2001. 

\bibitem{JV} M. Josuat-Verg\`es, \textit{A generalization of Euler numbers to finite Coxeter groups}, Annals of Combinatorics, to appear. 

\bibitem{M}  I.~G.~Macdonald, \textit{Symmetric functions and 
Hall polynomials} , Oxford University Press (1979).

\bibitem{RW} C. A. Robinson and S. Whitehouse, \textit{The tree 
representation of $\Sigma_{n+1}$},
J. Pure and Appl. Algebra, \textbf{111} (1-3) (1996), 245--253.

\bibitem{Sp} T. A. Springer, \textit{Remarks on a combinatorial problem}, Nieuw Arch. Wisk. (3) \textbf{19} (1971), 30-36.

\bibitem{St1} R.~P.~Stanley, \textit{Balanced Cohen-Macaulay complexes}, 
Trans. Amer. Math. Soc. \textbf{249} (1979), 139--157.

\bibitem{St2} \bysame, \textit{Some aspects of groups acting 
on finite posets}, J. Comb. Theory (A) \textbf{32}, No.~2 (1982),
132--161.

\bibitem{St3} \bysame, \textit{Enumerative 
Combinatorics, Vol.1}, Wadsworth and Brooks/Cole, Pacific Grove,
 CA (1986); Second edition, Cambridge Studies in Advanced Mathematics 49, Cambridge University Press, Cambridge, 2012.

\bibitem{St4} \bysame, \textit{Enumerative 
Combinatorics, Vol.2}, Cambridge Studies in Advanced Mathematics 62, Cambridge University Press, Cambridge, 1999.

\bibitem{St5} \bysame,  \textit{A survey of alternating permutations}, Contemporary Mathematics \textbf{531} (2010), 165--196.

\bibitem{St6} \bysame, \textit{Permutations}, Notes for the 
Amer. Math. Soc. Colloquium Lectures 2010.

\bibitem{Su1} S.~Sundaram, \textit{The homology representations of the 
symmetric group on Cohen-Macaulay subposets of the partition lattice},
Adv. in Math. \textbf{104} (2)(1994), 225--296.

\bibitem{SuJer}  \bysame,  \textit{Applications of the Hopf trace formula to computing homology representations}, in \textit{Jerusalem Combinatorics Conference 1993}, 
Contemporary Math., \textit{Barcelo and Kalai, eds.},  Amer. Math. Soc.\textbf{178} (1994), 277--309.

\bibitem{Su2} \bysame, \textit{The homology  of partitions with an even number of blocks}, J. Algebraic Comb. \textbf{4} (1995), 69--92.

\bibitem{Su3} \bysame, \textit{Plethysm, partitions with an even number of blocks and Euler numbers},
DIMACS Series in Discrete Mathematics and Theoretical Computer Science \textbf{24}, (Billera, Greene, Simion, Stanley, eds.), Amer. Math. Soc. (1996), 171--198.

\bibitem{Su4}  \bysame,  \textit{Homotopy of non-modular partitions and the Whitehouse module}, 
J. Algebraic Comb. \textbf{9} (1999), 251--269.

\bibitem{Su5}  \bysame,  \textit{On the topology of two partition posets with forbidden block sizes}, 
J. Pure and Applied Algebra \textbf{155} (2001), 271--304.  


\bibitem{SWa} S.~Sundaram and M.L. Wachs,
\textit{The homology representations of the $k$-equal partition lattice}, 
Trans. Amer. Math. Soc.  \textbf{349} (3) (1997), 935--954.

\bibitem{SWe}  S.~Sundaram and V. Welker,
\textit{Group actions on subspace arrangements and applications to 
configuration spaces}, Trans. Amer. Math. Soc., Vol. \textbf{349}, No. 4 (1997), 1389--1420.


\bibitem{Wa} M. L. Wachs, \textit{Poset Topology: Tools and Applications}, 
in \textit{Geometric Combinatorics}, IAS/Park City Math. Series Vol. \textbf{13}, Ezra Miller, 
Victor Reiner, Bernd Sturmfels, Editors, Amer. Math. Soc. and Inst. for Adv. Study (2007).

\bibitem{Wal} J. Walker, \textit{Homotopy type and Euler 
characteristic of partially ordered sets},
Europ. J. Combinatorics \textbf{2} (1981), 373--384.

\bibitem{Wel} V. Welker,   \textit{Partition Lattices, 
Group Actions on Subspace Arrangements and Combinatorics of 
Discriminants}, Habilitationsschrift, Department of Mathematics and 
Computer Science, GH Universit\"at Essen.

\bibitem{OEIS} OEIS Foundation Inc. (2011), \textit{The On-Line Encyclopedia of Integer Sequences}, http://oeis.org/A113897.

\end{thebibliography}

\bibliographystyle{amsplain.bst}

\end{document}